\documentclass[12pt]{amsart}
\usepackage{amssymb, amsmath, amsfonts}
\usepackage[raiselinks=false,colorlinks=true,citecolor=blue,urlcolor=blue,linkcolor=blue,bookmarksopen=true,pdftex]{hyperref}
\newtheorem{theorem}{Theorem}[section]
\newtheorem{proposition}[theorem]{Proposition}
\newtheorem{lemma}[theorem]{Lemma}
\newtheorem{remark}[theorem]{Remark}
\newtheorem{definition}[theorem]{Definition}

\newcommand{\wt}{\widetilde}
\newcommand{\im}{\textrm{Im}}

\begin{document}
\baselineskip=15.5pt
\title{Degrees of maps between locally symmetric spaces}  
\author[A. Mondal]{Arghya Mondal}
\author[P. Sankaran]{Parameswaran Sankaran}
\address{The Institute of Mathematical Sciences, CIT
Campus, Taramani, Chennai 600113, India}
\email{arghya@imsc.res.in}
\email{sankaran@imsc.res.in}

\date{}

\begin{abstract}   
Let $X$ be a locally symmetric space $\Gamma\backslash G/K$ where $G$ is a connected non-compact semisimple real Lie group with trivial centre, $K$ is a maximal compact subgroup of $G$, and $\Gamma\subset G$ is a torsion-free irreducible lattice in $G$.   Let $Y=\Lambda\backslash H/L$ be another 
such space having the same dimension as $X$.  Suppose that real rank of $G$ is at least $2$.   We show 
that any $f:X\to Y$ 
is either null-homotopic or is homotopic to a covering projection of degree  
an integer that depends only on $\Gamma$ and $\Lambda$.  As a corollary 
we obtain that the set $[X,Y]$ of homotopy classes of maps from $X$ to $Y$ is finite.

We obtain results on the (non-) existence of orientation reversing diffeomorphisms on 
$X$ as well as the fixed point property for $X$.
\end{abstract}
\maketitle
\section{Introduction}

Let $G$ be a connected non-compact semisimple real Lie group with trivial centre and without compact factors and let $K\subset G$ be a maximal compact 
subgroup. Then $G/K$ is a Riemannian globally symmetric space which is contractible.  If $\Gamma$ is a torsion-free lattice in $G$, then the locally symmetric space $X=\Gamma\backslash G/K$ (of non-compact type) is a manifold which is an Eilenberg-MacLane space $K(\Gamma,1)$.  
Locally symmetric spaces are fundamental objects which arise in different areas of mathematics 
such as geometry, Lie groups, representation theory, number theory as well as topology.  In this article 
we study the problem of homotopy classification of maps between two such spaces.  

Our main result is the following.   Recall that a lattice $\Gamma\subset G$ in a connected semisimple real Lie group 
is called {\it irreducible} if its image in $G/N$ under the projection $G\to G/N$ is dense for any non-compact normal subgroup $N$ of $G$.  (This definition is weaker than the one given in \cite{zimmer}; the two definitions agree when $G$ has 
no non-trivial compact factors.)  

Let $f:M\to N$ be any continuous map between two oriented connected manifolds of the same dimension $n$.
Recall that if $M$ and $N$ are compact, the $\deg(f)\in \mathbb{Z}$ is defined by $f_*(\mu_M)=\deg(f).\mu_N$ where $f_*:H_n(M;\mathbb{Z})=\mathbb{Z}\mu_M \to \mathbb{Z}\mu_N=H_n(N;\mathbb{Z})$ 
and $\mu_M,\mu_N$ are the fundamental classes of $M$ and $N$ respectively.  If $M$ and $N$ are non-compact and 
if $f$ is {\it proper}, then $\deg(f)\in \mathbb{Z}$ is defined in an analogous manner using 
compactly supported cohomology.

\begin{theorem} \label{higherrank}
Let $G, H$ be connected semisimple Lie groups with trivial centre and without compact factors and let $K,L$ be maximal compact subgroups of $G$ and $H$ respectively.   Suppose that the (real) rank of $G$ 
is at least $2$ and that $\dim G/K\ge \dim H/L$.   Let $\Gamma$ be an irreducible torsionless lattice in $G$ and let $\Lambda$ be any torsionless lattice in $H$. 
 There exists a non-negative integer $\delta=\delta(\Gamma,\Lambda)$ such that the following hold: 
Any continuous map $f:\Gamma\backslash G/K\to \Lambda\backslash H/L$ 
is either null-homotopic or is homotopic to a proper map $g$ such that $\deg(g)=\pm\delta$. 
\end{theorem}

 It will be shown that $\delta(\Gamma,\Lambda)$ equals $[\Lambda:\Lambda']$, the index of $\Lambda'\subset \Lambda,$ where $\Lambda'$ is {\it any} finite index subgroup of $\Lambda$ isomorphic to $\Gamma$.  If there is no such subgroup, then $\delta(\Gamma,\Lambda)=0.$   
 The assertion in the above theorem that $f$ is homotopic to a covering projection if 
 it is not null-homotopic follows by a straightforward argument using the rigidity theorem and the Margulis' normal 
 subgroup theorem.  By comparing the volumes of the locally symmetric spaces it can be seen that $\delta(\Gamma,\Lambda)$ is {\it independent} of the continuous map $f$.    When $\delta>0$, both $\delta$ and $-\delta$ occur as degrees of maps as in the theorem 
 only if $G/K$ admits an orientation reversing isometry.  See Remark \ref{reverse}.
 We classify, almost completely, all globally symmetric spaces $G/K$ with $G$ simple, such that for {\it any} lattice $\Gamma\subset G$, the locally symmetric space $X=\Gamma\backslash G/K$ 
does not admit an orientation reversing isometry.  The only cases which remain unsettled are certain $G/K$ where $G$ 
is exceptional.  

 We obtain the following result as a corollary.

\begin{theorem} \label{homotopy} 
Let $X=\Gamma\backslash G/K,Y=\Lambda\backslash H/L$ where $G,H,\Gamma, \Lambda$ satisfy the hypotheses Theorem \ref{higherrank}.   Then the set $[X,Y]$ of all  (free) 
homotopy classes of maps from $X$ to $Y$ is finite. 
\end{theorem}

Theorem \ref{higherrank} is not valid when $G$ has rank $1$.   
If $\Gamma$ is such that its abelianization is infinite, then there exist infinitely many 
non-trivial homomorphisms $\Gamma\to \Lambda$ with image an infinite cyclic group.  Using this, it is easily 
seen that there are infinitely many pairwise non-homotopic continuous maps of $\Gamma \backslash G/K$ to any irreducible locally symmetric space.   See \S\ref{firstbetti}.   However, we observe, in \S\ref{self}, that any {\it self-map} of a rank-$1$ locally symmetric space (of non-compact type) having non-zero 
degree is a homotopy equivalence and hence that, up to homotopy, there are only {\it finitely} many such maps provided 
$\dim G/K\ge 3$.  

Theorem \ref{homotopy} should be contrasted with the situation in the case of irreducible hermitian symmetric spaces 
of {\it compact type}.  It is known that there are continuous self-maps of the complex Grassmann manifold $\mathbb{C}G_{n,k}$  
having arbitrarily large degree.  (Recall that $\mathbb{C}G_{n,k}=SU(n)/S(U(k)\times U(n-k))$ is the space of all $k$-dimensional 
$\mathbb{C}$-vector subspaces of $\mathbb{C}^n$.) 
See \cite[Example 2.2, p. 208]{friedlander}.  

The paper is organised as follows: In \S\ref{index} we introduce 
the notions of $\mathcal{F}$-co-Hopficity and minimal index.  The main theorems are proved in \S\ref{proof}.  
In \S\ref{orient} we consider the problem of classifying locally symmetric spaces which admit an orientation 
reversing isometry.  
In \S\ref{rankone} we consider the 
case of rank-$1$ locally symmetric spaces.  Most of our results in this section hold for manifolds with 
negative sectional curvature.

{\it Throughout this paper, by a locally symmetric space, we mean a locally symmetric space of {\em non-compact} type.}


\section{$\mathcal{F}$-Co-Hopficity and minimal index}\label{index}
Recall that a group $\Gamma$ is said to be {\it residually finite} if, given any $\gamma\in \Gamma$, $\gamma\ne 1$, there exists a 
finite index subgroup $\Lambda$ such that $\gamma\notin \Lambda$.   A group $\Gamma$ is called {\it Hopfian} (resp. {\it co-Hopfian}) if any surjective (resp. injective) homomorphism $\Gamma\to \Gamma$ is an automorphism.  
Any finitely generated subgroup of a general linear group over a field is residually finite, and, any finitely generated residually finite group is Hopfian.  The latter result is due to Mal'cev. See \cite{ls}.  In particular, any lattice in a connected semisimple linear Lie group $G$, being finitely generated, is residually finite and hence Hopfian.   
A deep result of Sela \cite{sela} says that any torsion-free hyperbolic group is Hopfian.  In particular, the fundamental 
group of a compact negativey curved manifold is Hopfian.   
   
\begin{lemma}  \label{finiteindex}
Let $\Gamma$ be an infinite torsion-free group.\\
(i) Suppose that any non-trivial normal subgroup of $\Gamma$ has finite index 
in $\Gamma$. Let $\phi:\Gamma \to \Lambda$ be any surjective homomorphism where $\Lambda$ is infinite.
Then $\phi$ is an isomorphism. If $\Gamma$ is also co-Hopfian, then any non-trivial endomorphism of $\Gamma$  
is an isomorphism.  \\
(ii) (Cf. Hirshon \cite{hirshon}) Let $\phi:\Gamma\to \Gamma$ be an endomorphism where $\im(\phi)\subset \Gamma$ has finite index in $\Gamma$.  
If $\Gamma$ is finitely generated, residually finite and co-Hopfian, then $\phi$ is an automorphism.
\end{lemma}
\begin{proof}
(i)  Note that $\ker(\phi)$ has infinite index in $\Gamma$ since $\Lambda$ is infinite.  
By our hypothesis on $\Gamma$, it follows that $\ker(\phi)$ is trivial and so $\phi$ is an isomorphism.  

Let $\Gamma$ be co-Hopfian.  If $\phi:\Gamma\to \Gamma$ is a non-trivial endomorphism, then $\phi(\Gamma)$ 
is infinite as $\Gamma$ is torsion-free.  It follows from what has been 
shown already that $\phi$ is a monomorphism.  Since $\Gamma$ is co-Hopfian, we must have $\phi$ is onto and 
so $\phi$ is an automorphism. 

(ii)  This is essentially due to Hirshon \cite[Corollary 3]{hirshon} who showed, without the hypothesis of co-Hopficity property, that $\phi$ is a monomorphism.  The co-Hopficity of $\Gamma$ implies that $\phi$ is an automorphism.    
\end{proof}


\begin{definition}
Let $\Gamma, \Lambda$ be any two infinite groups.   Let $\delta(\Gamma,\Lambda)$ be defined as 
\[\delta(\Gamma,\Lambda):=\min[\Lambda: \Lambda']\] 
where the infimum is taken over all finite index subgroups $\Lambda'$ of $\Lambda$ which are isomorphic to 
$\Gamma$.   If there is no such subgroup, 
we set $\delta(\Gamma,\Lambda):=0$.  We call $\delta(\Gamma,\Lambda)$ the {\em minimal index} of 
$\Gamma$ in $\Lambda$. 
\end{definition}

Let $X$ be a finite CW complex which is a $K(\Lambda,1)$-space.  Note that if $\Lambda'$ is any finite index 
subgroup of $\Lambda$, then the coving space $X'$ of $X$ corresponding to the subgroup $\Lambda'$ 
is also a finite CW complex which is a $K(\Lambda',1)$-space.  Also the Euler characteristic $\chi(\Lambda):=\chi(X)$ is non-zero if and only if $\chi(\Lambda')$ is non-zero 
and $\chi(\Lambda)[\Lambda:\Lambda']=\chi(\Lambda')$.  It follows that  
in case $\chi(\Lambda)\ne 0$, then the minimal index equals the index:   
$\delta(\Lambda',\Lambda)=[\Lambda:\Lambda']$.

\begin{definition} 
We say that  $\Lambda$ is {\em $\mathcal{F}$-co-Hopfian} if $[\Lambda:\Lambda_1]=[\Lambda:\Lambda_2]$ for any two 
finite index subgroups $\Lambda_1,\Lambda_2\subset \Lambda$ such that $\Lambda_1\cong\Lambda_2$.   
\end{definition}

The group $\mathbb{Z}$ is not $\mathcal{F}$-co-Hopfian.      
A non-abelian free group of finite rank is $\mathcal{F}$-co-Hopfian but not co-Hopfian.  
More generally, we see from the discussion preceding the above definition that 
if there exists a $K(\Lambda, 1)$-space where $X$ is a finite CW complex with 
$\chi(X)\ne 0$, then $\Lambda$ is $\mathcal{F}$-co-Hopfian.  
Also if $\Lambda$ admits no non-trivial finite quotients, then $\Lambda$ is vacuously $\mathcal{F}$-co-Hopfian.

  We recall now a natural metric on a locally symmetric space.  
Let $\Theta:G\to G$ be an involutive automorphism with fixed group a maximal compact subgroup $K$. Let $\theta:\frak{g}\to \frak{g}$ be its differential where $\frak{g}:=Lie(G)$ and let $\frak{g}=\frak{k}\oplus \frak{p}$ be the Cartan decomposition 
where $\frak{k}=\{V\in \frak{g}\mid \theta(V)=V\}=Lie(K), \frak{p}=\{V\in \frak{g}\mid \theta(V)=-V\}$. 
The Killing form of $\frak{g}$ restricted to $\frak{p}$ is positive definite.  
Since the tangent bundle of $G/K$ is obtained as $G\times_K \frak{p}\to G/K$ (where $K$ acts on $\frak{p}$ via the 
adjoint),  this yields a $G$-invariant Riemannian metric 
on $G/K$ with respect to which it is a globally symmetric space.   The canonical Riemannian metric on $\Gamma\backslash G/K$ is obtained by requiring the covering projection $G/K\to \Gamma\backslash G/K=X$ to be a local isometry.   We refer 
to volume of $X$ with respect to this metric the {\it canonical volume}.
By the strong rigidity theorem of Mostow-Margulis-Prasad (\cite{eberlin}) the canonical volume of 
$X$ is a homotopy invariant. 
We have the following observation.

\begin{lemma} \label{fcohopf}  Let $\Lambda$ be an infinite torsion-free group.  
Suppose that any one of the following holds: (i) there exists a finite $K(\Lambda,1)$ complex and 
$\chi(\Lambda)\ne 0$;
(ii) $\Lambda$ is the fundamental group of a complete Riemannian manifold $M$ of finite volume with sectional curvature in $[-k_2,-k_1]$ where $0<k_1\le k_2<\infty;$  
(iii) $\Lambda$ is an irreducible lattice in a semisimple Lie group $H$ with trivial centre and having no (non-trivial) 
compact factors.  Then $\Lambda$ is $\mathcal{F}$-co-Hopfian.
 \end{lemma}
\begin{proof} 
By the above discussion, it only remains to consider cases (ii) and (iii).  

\noindent
Case (ii).   Let $\Lambda_1, \Lambda_2\subset \Lambda$ be finite index subgroups of $\Lambda$ such that $\Lambda_1\cong \Lambda_2$.  Let $M$ be a complete Riemannian manifold with sectional curvature 
in $[-k_2,-k_1]$ with $\pi_1(M)=\Lambda$.  It is known that the simplicial volume $||M||$ of such a manifold is 
non-zero and finite in view of Thurston's inequality.   See  \cite{inoue-yano} and \cite[\S 0.3]{gromov}.  
Let $p_j:M_j\to M$ be the covering of $M$ such that $\pi_1(M_j)=:\Lambda_j, j=1,2$.   Then $||M_j||=\deg(p_j).||M||$. 
Since $\Lambda_1\cong \Lambda_2$, we have $||M_1||=||M_2||$ and so 
$[\Lambda:\Lambda_1]=\deg(p_1)=\deg(p_2)=[\Lambda:\Lambda_2]$.   

\noindent
Case (iii).     This time we use the canonical volume instead 
of simplicial volume.   As remarked already, the Mostow-Margulis-Prasad rigidity theorem 
implies that the canonical volume of a locally symmetric space is a homotopy invariant.   Thus, proceeding as in case (ii), we see that the assertion holds.
\end{proof}

\begin{remark}{\em   (1)  Suppose that $\dim G/K >2$ or that $\Gamma$ is cocompact.   Suppose that $\Gamma_1$ 
is a finite index subgroup of $\Gamma$ and that $\Gamma_2\subset \Gamma$ is isomorphic to $\Gamma_1$ 
then $\Gamma_2$ is necessarily of finite index in $\Gamma$ and by the above lemma $[\Gamma:\Gamma_1]
=[\Gamma:\Gamma_2]$.  Indeed, 
since $\Gamma_2\subset \Gamma$ is discrete and is isomorphic 
to the lattice $\Gamma_1$, the main result of \cite{prasad76} asserts that  
$\Gamma_2$ is also a lattice.   
Hence $\Gamma_2$ has finite index in $\Gamma$. \\
(2)  It is known that the simplicial volume of a compact locally symmetric space of non-compact type is positive.  See \cite{lafont-schmidt}, \cite{buch}.    Thus the argument in the proof of case (ii) applies to case (iii) as well. \\
(3) It is known that the fundamental group of a complete negatively curved manifold $M$ with finite 
volume is co-Hopfian provided $\dim (M)\ge 3$.   See \cite[\S5]{belegradek}. 
}
\end{remark}

\begin{lemma} \label{deltavalue}  Let $\phi:\Gamma\to \Lambda$ be a homomorphism such that $[\Lambda:\phi(\Gamma)]<\infty$ where $\Gamma$ is  an infinite group in which every proper normal subgroup 
has finite index in $\Gamma$ and  $\Lambda$ is an infinite torsion-free $\mathcal{F}$-co-Hopfian group.  Then  
$\delta(\Gamma,\Lambda)=[\Lambda:\phi(\Lambda)]$. 
\end{lemma}
\begin{proof}  Since $\Lambda$ is torsion free, the image of $\phi$ is infinite.  Hence $\ker(\phi)$ has 
infinite index and hence, by our hypothesis on $\Gamma$, $\phi$ is a monomorphism.  
If $\Lambda_1$ is any finite index subgroup of $\Lambda$,  such that $\Lambda_1\cong \Gamma$, 
then $\Lambda_1\cong \phi(\Gamma)$.  
By the $\mathcal{F}$-co-Hopf property of $\Lambda$, we have $[\Lambda:\Lambda_1]=[\Lambda:\phi(\Gamma)]$.
The lemma is now immediate from the definition of $\delta(\Gamma,\Lambda)$.
\end{proof}

\section{Proof of the Main Theorems}\label{proof}
In this section we prove Theorems \ref{higherrank} and \ref{homotopy}.  

\noindent 
{\it Proof of Theorem \ref{higherrank}.}   
Without loss of generality we assume that $f$ preserves the base points, which 
are taken to be the identity double-cosets. Thus $\pi_1(X)=\Gamma, \pi_1(Y)=\Lambda$ (suppressing the base point in the notation). 

Suppose that $f$ is not null-homotopic. Then $f_*:\pi_1(X)=\Gamma\to \Lambda=\pi_1(Y)$ is non-trivial.  Note that 
$\Gamma$ and $\Lambda$ are torsion-free since $X$ and $Y$ are aspherical manifolds.  
Since $rank(G)\ge 2$ and $\Gamma$ is an irreducible lattice, by Margulis' normal subgroup theorem  (\cite[Chapter 8]{zimmer}) $\ker(f_*)$ is finite 
or $Im(f_*)$ is finite.  As $\Gamma$ and $\Lambda$ are torsionless and $f_*$ is non-trivial, we must have $\ker(f_*)$ is trivial.   Hence $f_*$ is an 
isomorphism of $\Gamma$ onto a subgroup of $\Lambda_1\subset \Lambda$.  Since $\Lambda_1\subset H$ is discrete, 
by  the main result of Prasad \cite{prasad76} we see that $\dim H/L\ge \dim G/K$.  Since $\dim X\ge \dim Y$ 
by hypothesis, we 
must have equality and, again by the same theorem, $\Lambda_1$ is a 
lattice in $H$.  Since $\Lambda$ is a lattice of $H$, we conclude that $\Lambda_1\subset \Lambda$ must have finite index in $\Lambda$.  It follows that $f:X\to Y$ factors as $p\circ f_1$ where $f_1:X\to Y_1$, and $p:Y_1\to Y$ is a finite 
covering projection, where $Y_1:=\Lambda_1\backslash H/L$.  Since $f_1$ induces isomorphism in fundamental groups, it is a homotopy 
equivalence as $X,Y_1$ are aspherical manifolds.  By the Mostow-Margulis rigidity theorem, it follows that $f_1$ is homotopic to an isometry $h$.  We let $g:=p\circ h$.

It remains to show that $\deg(g)=\pm\deg(p)=\pm[\Lambda:\Lambda_1]$ equals $\pm\delta(\Gamma,\Lambda)$, where the sign is positive if $f_1$ is orientation preserving and is negative otherwise.  
By Lemma \ref{fcohopf}, $\Lambda$ is $\mathcal{F}$-co-Hopfian.  Since $rank(G)\ge 2$, in view of Margulis' 
normal subgroup theorem, we see that the hypotheses of Lemma 
\ref{deltavalue} hold and we have $\delta(\Gamma,\Lambda)=[\Lambda:\Lambda_1]$.  This completes the 
proof. \hfill $\Box$

It is a well-known and classical result of Kodaira \cite[Theorem 6]{kodaira} that a compact locally Hermitian symmetric space has the structure of a 
smooth projective variety.   Any holomorphic map $f:X\to Y$ is automatically an {\it algebraic} morphism 
of varieties (by the GAGA principle).  

\begin{proposition} \label{morphism}  Let $X=\Gamma\backslash G/K, Y=\Lambda\backslash H/L$ be as in Theorem \ref{higherrank}.  Suppose that $X,Y$ are locally Hermitian symmetric spaces and that $h:X\to Y$ is a non-constant morphism of smooth projective varieties.  Then  $h$ is a covering projection.  
\end{proposition}
\begin{proof}     
Let $Z=h(X)$.    Then $Z$ is irreducible as a variety and has positive dimension.
We claim that there exists a subvariety $V\subset X$ such that $h(V)=Z$ and $\dim V=\dim Z. $ 
This is proved by induction on the dimension of $X$.     
If $\dim Z=\dim X$, then $h(X)=Z=Y$.  So assume that $\dim Z<\dim X$. 
Let $X'=X\cap \mathcal{H}\subset X$ be a generic hyperplane section with respect to an imbedding $X\subset \mathbb{P}^N$.     
 Then $h(X')=Z$.  (We need only choose the hyperplane $\mathcal{H}\subset \mathbb{P}^N$ as follows: 
 let $h(x)=z$ be a smooth point 
 of $Z$ and let $L=\ker (T_xX \to T_zZ)\subset T_xX\subset T_x\mathbb{P}^N$.  
 We choose $\mathcal{H}$ such that $T_x\mathcal{H}\subset T_x\mathbb{P}^N$ does {\it not} contain $L$.)
By induction there exists a subvariety $V\subset X'$ such that $\dim V=\dim Z>0$ and $h(V)=Z$.   

Now the fundamental cycle $\eta_V\in H_*(X)$ of $V$ maps to a non-zero multiple of 
$\eta_{Z}\in H_*(Y)$.   Since $\eta_{Z}\ne 0$ and since $\dim(Z)>0$, we conclude that $h$ is not null-homotopic.  
By Theorem \ref{higherrank},  $\deg(h)$ is non-zero and $h$ is surjective.  
Let $p:Y_1\to Y$ be the covering projection 
corresponding to the (finite index subgroup) $h_{*}(\pi_1(X))\subset \pi_1(Y)$.  Then $h$ factors as $h=p\circ h_1$
where $h_1:X\to Y_1$ is a homotopy equivalence and $p:Y_1\to Y$ a covering projection.   
Since $p$ is a local biholomorphism, $h_1:X\to Y_1$ is holomorphic and is a homotopy equivalence,  $h_{1*}:H_*(X)\to H_*(Y_1)$ is an isomorphism.  In particular $h_1$ is surjective.  
We claim that $h_1$ is one-one.   Clearly $h_1$ is birational (since it has degree $1$) and so the fibres of $h_1$ are connected by Zariski's main theorem (\cite[Corollary 11.4, Chapter III]{hartshorne}). If $h_1$ is not one-one, there would exist a point $y\in Y_1$ and a positive dimensional irreducible subvariety 
$C\subset h_1^{-1}(y)\subset  X$ which maps to the point $y$.  This implies that $h_{1*}(\eta_C)=0,$ a contradiction 
since $\eta_C\neq 0$.   Thus $h_1$ is a holomorphic bijection and hence a biholomorphism.  
\end{proof}

\begin{remark}  \label{reverse}  {\em Let $X=\Gamma\backslash G/K, Y=\Lambda\backslash H/L$ be as in 
Theorem \ref{higherrank}. 
Suppose that $f:X\to Y$ is a (proper) continuous map having degree $\delta>0$.
If $X$ admits an orientation reversing isometry $h:X\to X,$ 
then $\deg(f\circ h)=-\delta$ and so both $\delta$ and $-\delta$ occur 
as degrees of maps between $X$ and $Y$.  An analogous statement holds if $Y$ admits an orientation reversing 
isometry. On the other hand suppose that $g:X\to Y$ has degree $-\delta.$
Then replacing $f$ and $g$ up to homotopy if necessary, we may assume without loss of generality that $f,g:X\to Y$ 
are covering projections (which preserve base points).  
The isomorphisms, $\Gamma \cong f_*(\Gamma)$ and  
$\Gamma\cong g_*(\Gamma)$,  are induced by isometries $f': X\to f^*(\Gamma)\backslash H/L$ and $g':X\to g_*(\Gamma)\backslash H/L$ which lift to isometries $\wt{f}',\wt{g}':G/K\to H/L$ respectively. 
It is clear that $f'$ is orientation preserving and $g'$ is orientation reversing.  Hence $h:=\wt{g}'^{-1}\circ \wt{f}'$ is 
an orientation reversing isometry of $G/K$.}
\end{remark}

In the case when $rank(G)>rank(H)\ge 1$ we have the following result.  

\begin{theorem}
We keep the notations as in Theorem \ref{higherrank}.  
(i) If $rank(G)>rank(H)$, then any continuous map $f:X\to Y$ is null-homotopic.
(ii) Suppose that $rank(G)>1$ and that $M$ is any closed negatively curved manifold whose 
sectional curvature lies in $[-k_2,-k_1]$ for some $0<k_1<k_2$, then any continuous map $f:X\to M$ is null-homotopic. 
\end{theorem}
Note that, unlike in Theorem \ref{higherrank}, we do not require that $\dim X \ge \dim Y$. 
\begin{proof}
We need only show that the induced map between the fundamental groups 
is trivial. 
As in the proof of Theorem \ref{higherrank}, we see that, if $f$ is not null-homotopic, then 
$f$ induces a {\it monomorphism} of fundamental groups. We shall show in each case that 
this is not possible. 

Case (i): A result of Prasad and Raghunathan \cite{prasad-raghunathan} 
says that $\Gamma$ contains a free abelian group of rank equal to the rank of $G$.  
When $\Gamma$ is cocompact, this is due to Wolf \cite{wolf}.
This is the 
maximum possible rank---denoted $rank(\Gamma)$---of any abelian subgroup of $\Gamma$.     
Thus by hypothesis $rank(\Gamma)>rank(\Lambda)$.  It follows that $f_*:\Gamma\to \Lambda$ 
is {\it not} a monomorphism. 

Case (ii): Our hypothesis on $M$ implies that 
$\pi_1(M)$ is a hyperbolic group and hence has rank $1$.  The same argument as in case (i) applies and we conclude that $f_*:\Gamma \to \pi_1(M)$ cannot be  
a monomorphism.  
\end{proof}

Next we turn to the proof of Theorem \ref{homotopy}.

\noindent
{\it Proof of Theorem \ref{homotopy}.}  
First note that there are only finitely many subgroups of $\Lambda$ having index $\delta(\Gamma,\Lambda)$. 
Corresponding to any such group $\Gamma\cong\Lambda_1\subset\Lambda$, the deck transformation group 
of the covering projection $p:Y_1\to Y$ is finite; here $Y_1$ corresponds to the subgroup $\Lambda_1$.   
Identifying $Y_1$ with $X$, the set of homotopy self-equivalences of $X$ equals the 
group $Out(\Gamma)$ of all outer automorphisms of $\Gamma=\pi_1(X)$.  It is a well-known fact that 
this latter group is finite.  For example, using the Mostow-Margulis-Prasad rigidity theorem, one has a 
natural homomorphism $Out(\Gamma)\to Aut(G)/G$  with kernel $N_\Gamma/\Gamma$ where $N_\Gamma$ denotes 
the normalizer of $\Gamma$ in $G$.  
Since $Aut(G)$ has only finitely many components, 
$Aut(G)/G$ is finite.   Also $N_\Gamma$ is a lattice in $G$ by a result of Borel \cite[Chapter V]{raghunathan1}.   Hence $Out (\Gamma)$ 
is finite.  It follows that $[X,Y]$ is finite.
\hfill $\Box$

\subsection{Fixed point property}
Recall that a topological space $X$ has the fixed point property if any continuous self map of $X$ 
has a fixed point.   Let $X$ be a compact connected smooth manifold such that its Euler characteristic $\chi(X)$ vanishes. 
Then the identity map of $X$ is homotopic to a diffeomorphism which is fixed point free.  (One need only 
consider the flow, for small values of time, of a nowhere vanishing smooth vector field on $X$.) 
Suppose that $G$ is a connected semisimple non-compact Lie group having rank at least $2$ such $Aut(G)=G$ and that $\Gamma$ is an  irreducible cocompact lattice such that $\Gamma=N_\Gamma$, the normalizer of $\Gamma$ in $G$.    
Denote by $U$ 
a maximal compact subgroup of the complexification of $G$ which contains $K$; thus $U/K$ is the compact dual of $G/K$. 
Suppose that $\chi(U/K)\ne 0$.  (This is equivalent to the requirement that $rank(U)=rank(K)$  or that $G$ admit a compact Cartan subgroup.) 
Then $X=\Gamma\backslash G/K$ has the fixed point property.  Indeed our hypothesis implies that $Out(\Gamma)$ is trivial.   
If $f:X\to X$ is a self-map, then either $f$ is null-homotopic, in which case the Lefschetz number $L(f)=1$, or, $f$ is homotopic to the identity map of $X$ in which case $L(f)=\chi(X)$.  By the Hirzebruch proportionality principle (cf. \cite{ko}), $\chi(X)$ is a non-zero multiple of $\chi(U/K)$.  Therefore, by the Lefschetz fixed point theorem $f$ has a fixed point.

If $N_\Gamma\ne \Gamma$ and if $N_\Gamma$ is torsionless, then we have a covering projection $p:X\to Y$, inducing the inclusion $\Gamma\subset N_\Gamma$ between fundamental groups, where $Y=N_\Gamma\backslash G/K$.  The deck transformation 
group $N_\Gamma/\Gamma$ being non-trivial, we see that $X$ admits a fixed point free diffeomorphism. 


\section{Existence of orientation reversing isometries}  \label{orient}
Let $G$ be a connected non-compact real semisimple  
Lie group without compact factors and let $K$ a maximal connected compact subgroup of $G$.  We will 
assume that $G$ has trivial centre but we do not assume that $G$ is of higher rank. Let $\Gamma$ be any torsionless discrete subgroup of $G$.  Then $X:=\Gamma\backslash G/K$ is an orientable manifold.  
In this section we address the questions: {\it does $X$ admit an orientation 
reversing isometry?}   The relevance of this question to degrees of maps between locally symmetric spaces with domain $X$ has already been explained Remark \ref{reverse}.
Note that if $X$ admits an orientation reversing isometry, then so does the universal cover $G/K$.    

We will assume 
that $\dim (X)\ge 3$ since the case of surfaces is well understood.   
We denote by $\wt{G}$ the group of all automorphisms of $G$.  
Then $Out(G)=\wt{G}/G$, the group of components of $\wt{G}$, is finite and the existence of an orientation 
reversing isometry on $G/K$ is equivalent to the surjectivity of the homomorphism $\omega: Out(G)\to \mathbb{Z}/2\mathbb{Z}$ whose kernel consists precisely of those components of $\wt{G}$ which consists of orientation preserving isometries.

We 
denote by $\wt{\Gamma}$ the normalizer of $\Gamma$ in $\wt{G}$.  If $\Gamma$ is a lattice in $G$, 
then from Borel's theorem we see that $\wt{\Gamma}$ is a lattice in $\wt{G}$
and $Out(\Gamma)\cong \wt{\Gamma}/\Gamma$.  
In particular $\wt{\Gamma}/\Gamma$ is a finite group.   In any case we have  
have a natural homomorphism $\eta:\wt{\Gamma}/\Gamma\to \wt{G}/G=Out(G)$.   Hence $X$ admits 
an orientation reversing isometry if and only if $Im(\eta)$ is not contained in $\ker(\omega)$, that is, 
$\wt{\Gamma}\cap C$ is non-empty for some component $C$ of $\wt{G}$ such that $\omega(C)\ne 0\in \mathbb{Z}/2\mathbb{Z}$.

Let $U/K$ be the simply connected compact dual of $G/K$. 
Suppose that $\sigma_u:U/K\to U/K$ is an isometry which we assume, without loss of generality, fixes the identity coset $o.$  As $U/Z(U)$ is covered by the identity component of the group of isometries of $U/K$,  $\sigma_u$ induces 
an automorphism, again denoted $\sigma_u,$ of the Lie algebra $Lie(U)=:\frak{u}=\frak{k}\oplus i\mathfrak{p}$ that stabilizes $\mathfrak{k}$ and hence $\frak{p}_*:=i\mathfrak{p}$ as well.   (Recall that $Lie(G)=\mathfrak{g}=\mathfrak{k}\oplus \mathfrak{p}$.)  
Let $\sigma$ be the complex linear extension of $\sigma_u$ to $\frak{u}\otimes \mathbb{C}=:\mathfrak{g}_\mathbb{C}$. 
Then $\sigma(\mathfrak{g})=\frak{g}$ and so $\sigma_0:=\sigma|_\mathfrak{g}$ is an autmorphism of $\mathfrak{g}$.  
We denote by the same symbol $\sigma_0$ the automorphism of $G/K$ induced by $\sigma_0\in Aut(\frak{g})$.
Conversely, starting with an isometry $\sigma_0$ of $G/K$ that fixes the identity coset of $G/K$, which is again denoted $o$,  we obtain an isometry of $U/K$ that fixes $o\in U/K$ (using the assumption that $U/K$ is simply connected).  

We have the natural isomorphism of tangent spaces $T_oU/K=\frak{u}/\frak{k}\cong\mathfrak{p}_*$ and $T_oG/K=\frak{g}/\frak{k}\cong \mathfrak{p}$.  
Note that $d\sigma_u:T_oU/K=\frak{p}_* \to \frak{p}_*=T_oU/K$ and $d\sigma_o:T_oG/K=\frak{p}\to \frak{p}=T_0G/K$ are  
restrictions of the same complex linear map $\sigma:\frak{g}_\mathbb{C}\to \frak{g}_\mathbb{C}.$   
Hence 
$\sigma_u$ is orientation preserving if and only if $\sigma_0$ is orientation preserving.

\begin{proposition}  \label{orduality}
Suppose that $U/K$ is the simply connected compact dual of $G/K$.  The space $G/K$ admits an orientation reversing 
isometry if and only if $U/K$ admits an orientation reversing isometry.  In particular, if $U/K$ does not admit any 
orientation reversing isometry, neither does $X=\Gamma\backslash G/K$ for any torsionless discrete subgroup  $\Gamma\subset G$. \hfill $\Box$
\end{proposition}

In view of the above proposition, we need only consider simply connected symmetric spaces $U/K$ of compact type 
to decide whether $G/K$ admits an orientation reversing isometry.   
We shall assume that $G$ is simple 
and settle this question completely for all symmetric spaces $G/K$ where $G$ is of classical type or $G$ is a complex Lie group.  We shall also address a few cases where $G$ is exceptional.  

Recall that if a smooth compact manifold admits an 
orientation reversing diffeomorphism, then the manifold represents either the trivial element or an element of order $2$ in the oriented cobordism ring $\Omega_*$ and hence all its Pontrjagin numbers are zero. See \cite[p. 186]{ms}.   Suppose that 
some Pontrjagin number of $U/K$ is non-zero, then it does not admit any orientation 
reversing diffeomorphism and the same is true of $X=\Gamma \backslash G/K$ as well for any torsionless discrete subgroup $\Gamma.$ 
In the case when 
$\Gamma$ is a cocompact lattice, by the Hirzebruch proportionality principle (\cite{hirzebruch56},  \cite{ko}) the corresponding Pontrjagin number of $X$ is non-zero and $X$ represents an element of infinite (additive) order in $\Omega_*$.

\subsection{Symmetric spaces of type IV}
When $G$ is a simply connected complex simple Lie group, we have $U=K\times K$, $G/K$ is of type IV, and $U/K\cong K$ is of type II.  
Let $\phi$ be an automorphism of $\mathfrak{g}=\mathfrak{k}+i\mathfrak{k}=:\mathfrak{k}^\mathbb{C}$ which preserves $\mathfrak{k}$.  Then either $\phi$ is the 
$\mathbb{C}$-linear extension $\theta:=\psi\otimes \mathbb{C}$ of an automorphism $\psi$ of $\mathfrak{k}$ or
is of the form $\sigma_0\circ \theta$ 
where $\sigma_0$ is the complex conjugation.  Since $\sigma_0$ restricted to 
$\mathfrak{p}=i\mathfrak{k}$ equals $-id$, $G/K$ admits an orientation reversing isometry if 
$\dim_\mathbb{C}\mathfrak{g}=\dim \mathfrak{k}$ is odd.   If $\phi=\theta$, then $\phi|i\mathfrak{k}$ is orientation 
reversing if and only if $\psi$ is.   Suppose that $\mathfrak{t}\subset \mathfrak{k}$ is the Lie algebra of a maximal torus $T\subset K$.  By composing with an inner automorphism of $\mathfrak{k}$ if necessary, we may assume 
that $\psi$ stabilizes $\mathfrak{t}$.  Let $\Delta$ be the set of roots of $\mathfrak{g}$ with respect to $\mathfrak{t}^\mathbb{C}$ and 
let $\Sigma$ be the set of simple roots for a positive system of roots $\Delta^+$.  Then $\Delta$ is also the set of 
roots of $\mathfrak{k}$ with respect to $\mathfrak{t}$.  
We may further assume, by composing with an inner automorphism of $\mathfrak{g}$ representing a suitable 
element of the Weyl group of $(\mathfrak{k},\mathfrak{t})$, that $\psi$ preserves $\Delta^+$. See \cite[Theorem 3.29, Ch. X]{helgason}.   
Then $\psi$ induces an automorphism of the Dynkin diagram of $(\mathfrak{k},\mathfrak{t})$.  
{\it We claim that $\psi$ is orientation preserving if and only if it induces an even permutation of the set of 
nodes of the Dynkin diagram, namely, $\Sigma$. } 
To see this, for any complex linear form $\gamma$ on  $\mathfrak{t}^\mathbb{C}$, 
let $H_\gamma\in \mathfrak{t^\mathbb{C}}$ be the unique element $\mathfrak{t}^\mathbb{C}$ such that $B(H,H_\gamma)=\gamma(H), \forall H\in \mathfrak{t}^\mathbb{C}.$  (Here $B(.,.)$ is the Killing form.) 
Then $iH_\alpha\in\mathfrak{t}$ for all $\alpha\in \Delta.$  Let $X_\alpha, \alpha\in \Delta,$ be a Weyl basis 
(in the sense of \cite[Definition, p.421]{helgason}).  Then, for any $\beta\in \Delta^+$, 
$X_\beta-X_{-\beta}, i(X_\beta+X_{-\beta})$ form an $\mathbb{R}$-basis for the real vector space 
$\mathfrak{k}\cap (\mathfrak{g}_\beta+\mathfrak{g}_{-\beta})=:\mathfrak{k}_\beta.$  The $\mathfrak{k}_\beta,\beta\in \Delta^+,$ 
are the non-trivial irreducible $T$-submodules of $\mathfrak{k}$ and we have $\psi(\mathfrak{k}_\beta)=\mathfrak{k}_{\gamma}$ where $\gamma=\psi^*(\beta)=\beta \circ\psi\in \Delta^+$.  In fact, the assumption that $\psi(\Delta^+)=\Delta^+$ implies that the matrix of $\psi:\mathfrak{k}_\beta\to \mathfrak{k}_{\gamma}$ with respect to the ordered bases of $\mathfrak{k}_\beta, \mathfrak{k}_\gamma$ 
as above is of the form $\left( \begin{smallmatrix}
x_\beta &-y_\beta \\y_\beta & x_\beta\end{smallmatrix}\right)$ with $x_\beta^2+y_\beta^2=1$ (see \cite[\S5, Ch. IX]{helgason}).  It follows that 
the $\psi:\mathfrak{k}\to \mathfrak{k}$ is orientation preserving if and only if $\psi|\mathfrak{t}$ is orientation 
preserving.  Since the basis $iH_\alpha, \alpha\in \Sigma,$ of $\mathfrak{t}$ is permuted by $\psi$, $\psi |\mathfrak{t}$ 
is orientation preserving if and only if $\psi^*:\Sigma\to \Sigma$ is an even permutation.  This proves our claim.
An inspection of the Dynkin diagrams reveals that an orientation reversing automorphism of $\mathfrak{k}$ exists 
precisely when $\mathfrak{k} =\frak{su}(n), n\equiv 0, 3\mod 4,$ or $\frak{k}=\frak{so}(2n), n\ge 4$.  We have proved

\begin{theorem}
An irreducible globally symmetric space $G/K$ of type VI admits an orientation reversing 
isometry if and only if either $\dim_\mathbb{C}G=\dim K$ is odd, or, $K$ 
is locally isomorphic to one of the groups $SU(4n+3), n\ge 0,$ and $SO(4m), m\ge 1$. 
\end{theorem}

From now on, till the end of \S4, we assume that $G$ is a connected real simple Lie group which is not a complex Lie group. Thus $G/K$ is 
of type III and $U$ is simple.  The results are tabulated in \S \ref{rptable}.

\subsection {Hermitian symmetric spaces} \label{hermitian} 
Let $U/K$ be simply connected compact irreducible Hermitian symmetric space $U/K$, where $U$ is simply connected and simple.   
There are six families of such spaces: AIII, DIII, BDI(rank $2$), CI, EIII, and EVII, using the standard notations (as in 
 \cite{helgason}).

There are exactly two invariant complex structures $\mathcal{C}$ and $\mathcal{C}'$ which are 
conjugate to one another in the sense that the complex structure on $\mathfrak{p}_*=T_oU/K$ induced by $\mathcal{C}$ and $\mathcal{C}'$ are complex conjugate of each other. 
These two complex structures are related by an automorphism $\sigma_u$ of $U$ that stabilizes $K$ thereby inducing 
the complex conjugation on $\mathfrak{p}_*$ (with respect to $\mathcal{C}$, say). 
In particular, the isometry $\sigma_u$ of $U/K$ is orientation reversing 
if and only if $\dim_\mathbb{C} U/K$ is odd.  See  \cite[Remark 2, \S 13]{bh}.   In particular, 
$E_7/U(1)\times E_6$, which has dimension $27$ (over $\mathbb{C}$), admits an orientation 
reversing isometry.
 
When $\dim_\mathbb{C}U/K=d$ is even, we shall show that 
an appropriate Pontrjagin number of $U/K$ is non-zero.   

The signature of a compact irreducible Hermitian symmetric space are known; see \cite[p. 163]{hirzebruch}.  
In the case of the complex Grassmann manifold (Type AIII), $\mathbb{C}G_{p+q,p}=SU(p+q)/S(U(p)\times U(q))$, the signature equals ${\lfloor (p+q)/2\rfloor}\choose {\lfloor p/2\rfloor}$  
when $d=pq$ is even.  (Cf. Shanahan \cite[p. 489]{shanahan}.)  The signature equals $4\lfloor p/2\rfloor$ for the complex quadric (Type BD I, rank $2$),  $SO(2+p)/(SO(2)\times SO(p)), p>2$.   
In the case of  type E III, namely $E_6/(Spin(10)\times U(1))$, the signature equals $3$.  
We shall presently show that when $\dim_\mathbb{C} SO(2p)/U(p)=p(p-1)/2=d$ (resp. $\dim_\mathbb{C} Sp(n)/U(n)=n(n+1)/2=d$) is even, $p_1^{d/2}\ne 0$ where $p_1$ is the first Pontrjagin class. 

We shall now compute $p_1:=p_1(SO(2p)/U(p))$.  For this purpose we shall use the notation and the formula for the total Chern class of $SO(2p)/U(p)$ given in \cite[\S 16.3]{bh} with respect to an $SO(2p)$-invariant complex structure 
compatible with the usual differentiable structure on $SO(2p)/U(p)$.  Let $\sigma_j=\sigma_j(x_1,\ldots, x_p)$ denote the $j$th symmetric polynomial in the indeterminates $x_1,\ldots, x_p$.  We set $\deg(x_j)=2$ and consider the graded polynomial algebra $\mathbb{K}[\sigma_1,\ldots, \sigma_p]$.  
If $\mathbb{K}$ is any field of characteristic other than $2$, the cohomology algebra $H^*(SO(2p)/U(p);\mathbb{K})$ is isomorphic to $\mathbb{K}[\sigma_1,\ldots,\sigma_p]/I$ where $I$ is the ideal generated by the elements 
$\lambda_j:=\sigma_j(x_1^2, x_2^2,\ldots, x_p^2), 1\le j<p$ and $\sigma_p$.  We take $\mathbb{K}$ to be $\mathbb{Q}$.

We have 
$c(SO(2n)/U(n))=\prod_{1\le i<j\le p} (1+x_i+x_j)$ from which we obtain the following formula for the 
total Pontrjagin class:
\[1-p_1+p_2-\cdots =\prod_{1\le i<j\le p}(1+x_i+x_j)\prod_{1\le i<j\le p}(1-x_i-x_j),\]
equivalently, 
\[p(SO(2n)/U(n))=\prod_{1\le i<j\le p}(1+(x_i+x_j)^2).\]
Therefore $p_1=\sum_{1\le i<j\le p} (x_i+x_j)^2=(p-1)\sum_{1\le i\le p}x_i^2+2\sum_{1\le i<j\le p}x_ix_j
=(n-1)\lambda_1+2\sigma_2 =2\sigma_2
=\sigma_1^2$, since $\lambda_1\in I$ and $\sigma_1^2=\lambda_1+2\sigma_2$. 
Since $SO(2n)/U(n)$ is K\"ahler, and since $H^2(SO(2n)/U(n);\mathbb{Q})= \mathbb{Q}\sigma_1$, we have 
$\sigma_1^{d}\ne 0$ where $d=\dim_\mathbb{C} SO(2p)/U(p)=p(p-1)/2$.  Thus we 
see that $p_1^{d/2}[SO(2p)/U(p)]\ne 0$ when $d$ is even.  

An entirely analogous computation shows that $p_1^{d/2}[Sp(p)/U(p)]\ne 0$ when $\dim_\mathbb{C}Sp(p)/U(p)=
n(n+1)/2=:d$ is even, using the formula for the total Chern class of $Sp(p)/U(p)$ as given in \cite[\S 16.4]{bh}.
In fact, in this case the computations can be carried out in the integral cohomology ring. 

To summarise we have proved, in view of Proposition \ref{orduality}, the following.

\begin{theorem}
An irreducible hermitian symmetric domain $G/K$ (or its simply connected compact dual $U/K$)  
admits an orientation reversing isometry if and only if its complex dimension is odd. \hfill $\Box$ 
\end{theorem}

\subsection{Oriented Grassmann manifolds}  \label{ogm}
The oriented Grassmann manifold $\wt{G}_{m+n,n}=SO(m+n)/SO(m)~\times SO(n)$ is dual to 
$SO_0(m,n)/SO(m)\times SO(n)$.  We leave out the well-known case of sphere, $\min \{m,n\}=1$, and also $\wt{G}_{4,2}\cong S^2\times S^2$, 
by assuming $m,n>1$ and $m+n>4$.  When the dimension $mn$ is odd, the symmetric space 
$\wt{G}_{m+n,n}$ admits an orientation reversing isometry.  So assume that $mn$ is even.

Shanahan \cite{shanahan} has shown that the signature of $\wt{G}_{m+n,n}$ equals ${\lfloor (m+n)/4\rfloor}\choose {\lfloor n/4\rfloor}$ when both $m,n$ are even, $mn\equiv 0\mod 8$, and is zero otherwise. 
 When $m\equiv n\equiv 2\mod 4$ and $m\ne n$, it was shown that $p_1^{mn/4}[\wt{G}_{m+n,n}]\ne 0$ in the proof of 
 \cite[Theorem 3.2]{sv}. Consequently, $\wt{G}_{m+n,n}$ does not admit an orientation reversing 
 diffeomorphism in these cases. 
 
It remains to consider the cases $\wt{G}_{m+n,n}=U/K$, $U=SO(m+n), K=SO(m)\times SO(n)$ where 
at least one of the numbers $m,n$ is odd, or,  $m=n\equiv 2 \mod 4$. 
Suppose that $m$ is odd (the case $n$ odd being analogous). Consider the isometry $\sigma_u:\wt{G}_{m+n,n}\to \wt{G}_{m+n,n}$ defined as 
$\sigma_u(xK)=D^{-1}xDK$ where 
$D:=D_{m+n}=diag(1,\ldots, 1, -1)\in O(m+n)$.  Note that $D^{-1}KD=K$.  Also the differential of $\sigma_u$ at $o\in U/K$ is the linear map of $\mathfrak{p}_*=
\{\left(\begin{smallmatrix} 0 &B \\-B^t &0\end{smallmatrix} \right )\mid B\in M_{m\times n}(\mathbb{R})\}\cong M_{m\times n}(\mathbb{R})$ defined by $B\mapsto BD_n$.   Since $m$ is odd, we conclude that $\sigma_u$ is orientation 
reversing.

Finally, let $n\equiv 2 \mod 4$ and consider $\sigma_u:\wt{G}_{2n,n}\to \wt{G}_{2n,n}$ defined as $uK\mapsto J^{-1}uJK=JuJK$ is an isometry 
where $J\in SO(2n)$ the matrix $J:=\left(\begin{smallmatrix} 0 & I_n\\ I_n&0\end{smallmatrix} \right)$. 
Note that $J^{-1}KJ= K$.    

Then the differential of $\sigma_u$ at 
$o\in U/K$ is the linear isomorphism of $\mathfrak{p}_*\cong M_n(\mathbb{R})$ given by $B\mapsto -B^t$, which is orientation reversing if and only 
if ${n}\choose{2}$ is odd.  Since $n\equiv 2\mod 4$, we conclude that $\sigma_u$ is orientation reversing. 
(Cf. \cite[Lemma 3.5]{sv}.)

\subsection{Quaternionic Grassmann manifolds} \label{qgm} 
Next consider the quaternionic Grassmann manifold $\mathbb{H}G_{p+q,p}=Sp(p+q)/Sp(p)\times Sp(q)$.  
S. Mong \cite{mong}  has shown that the signature of $\mathbb{H}G_{p+q,p}$ is zero if and only if both $p,q$ are odd.  
We claim that the Pontrjagin number $p_1^{pq}[\mathbb{H}G_{p+q,p}]\ne 0$ when $p\ne q$.   When $p=1<q$, 
$\mathbb{H}G_{q+1,1}\cong \mathbb{H}P^q,$ the quaternionic projective space and the result is due to Borel and Hirzebruch \cite[\S15.5]{bh}. 
The general case can be reduced the case of quaternionic projective space.
To see this, we 
assume, without loss of generality, that  $q>p>1$ and use the 
natural `inclusion' of the quaternionic projective space 
$j: \mathbb{H}P^{q-p-1}\subset \mathbb{H}G_{p+q,p}$ induced by the obvious inclusion of 
$H:=\left(\begin{smallmatrix} I_{p-1} &0&0\\
0& Sp(q-p+2) &0\\
0& 0& I_{p-1}\\ \end{smallmatrix} \right)\cong Sp(q-p+2)$ into  
$Sp(p+q)$ so that $H\cap (Sp(p)\times Sp(q))=Sp(1)\times Sp(q-p+1)\subset Sp(q-p+2)$.  If we view $\mathbb{H}G_{p+q,p}$ as the space of all 
$q$-dimensional left $\mathbb{H}$-vector spaces in $\mathbb{H}^{p+q}$ and $\mathbb{H}P^{q-p+1}=\mathbb{H}G_{q-p+2,1}$ as the space of $1$-dimensional $\mathbb{H}$-vector spaces in $\mathbb{H}e_{p}+\cdots 
+\mathbb{H}e_{q+1}$  then $j(L)=L+\mathbb{H}e_{1}+\cdots 
+\mathbb{H}e_{p-1}$ for $L\in \mathbb{H}P^{q-p+1}$. (As usual, $\mathbb{H}$ stands for the skew field of quaternions and 
$e_1,\ldots, e_{p+q},$ the standard basis for $\mathbb{H}^{p+q}$.)
Then the normal bundle to the imbedding $j$ is trivial, by using, for example, the description of the tangent bundle of $\mathbb{H}G_{p+q,p}$  due to Lam \cite{lam}.
So $j^*(p_1(\mathbb{H}G_{p+q,p}))=p_1(\mathbb{H}P^{q-p+1})\ne 0$ since $q-p+1\ge 2$. This shows that $p_1(\mathbb{H}G_{p+q,p})\ne 0$.   
Observing that the integral cohomology rings of $\mathbb{C}G_{p+q,q}$ and $\mathbb{H}G_{p+q,p}$ are isomorphic by an isomorphism that doubles the degree and using the fact that  $c_1^{pq}\ne 0$ where 
$c_1\in H^2(\mathbb{C}G_{p+q,p};\mathbb{Z})\cong \mathbb{Z}$ is a generator, we see that $p_1^{pq}\ne 0$ in $H^{4pq}(\mathbb{H}G_{p+q,p};\mathbb{Z})$.
Thus the Pontrjagin number $p_1^{pq}[\mathbb{H}G_{p+q,p}]\ne 0$. (See also \cite[Theorems 3.2(iii) and 3.3(ii)]{sv}.)   Hence when $p\ne q$, $\mathbb{H}G_{p+q,p}$ does not admit an orientation 
reversing diffeomorphism. 

It remains to consider the case $p=q\equiv 1$ mod $2$.  In this case 
$\mathbb{H}G_{2p,p}$ admits an orientation reversing isometry as we shall now show.   
For this purpose, we shall use the description 
$\frak{sp}(2p)=\{\left(\begin{smallmatrix} A & Z \\-{}^t\bar{Z} & B\end{smallmatrix}\right)\mid 
Z=\left(\begin{smallmatrix} Z_1 &Z_2\\-\bar{Z}_2& \bar{Z}_1\end{smallmatrix}\right),
Z_1,Z_2\in M_{p}(\mathbb{C}), A, B\in \mathfrak{sp}(p)\}$, where $\mathfrak{sp}(p)=\{\left(\begin{smallmatrix} X & Y\\-\bar{Y} & \bar{X}\end{smallmatrix}\right)
\mid X\in \mathfrak{u}(p), {}^tY=Y\in M_p(\mathbb{C})\}$. Thus $\mathfrak{p}_*$ consists of all matrices of the form 
$\left(\begin{smallmatrix} 0&Z\\-{}^t\bar{Z}&0\end{smallmatrix}\right)\in \mathfrak{sp}(2p)$. 
Conjugation by $J=\left(\begin{smallmatrix} 0& I \\I&0\end{smallmatrix}\right)$ is an automorphism of $\mathfrak{sp}(2p)$ which maps 
$\left(\begin{smallmatrix} A&Z\\-{}^t\bar{Z} &B\end{smallmatrix}\right)$
to $ \left(\begin{smallmatrix} B&-{}^t\bar{Z}\\ Z &A\end{smallmatrix} \right)$.  Evidently  it stabilizes $\mathfrak{k}=\mathfrak{sp}(p)\times \mathfrak{sp}(p)$ and, since $p$ is odd, reverses the orientation on $\mathfrak{p}_*$.

\subsection{Other symmetric spaces of classical type}
Consider the space $SU(n)/SO(n), n>2$ (type AI).  Note that $\frak{p}_* =iSym^0_n(\mathbb{R})$, consisting of 
trace $0$ symmetric matrices with purely imaginary entries.  Thus $\frak{p}_*\cong Sym^0(\mathbb{R})$.    
Conjugation by $D:=diag(1,\ldots, 1, -1)$ yields an isometry $\sigma_u$ 
defined as $xK \mapsto DxDK$ which induces on $\frak{p}_*$ the map $X\mapsto DXD$.   It is readily 
seen to be orientation reversing if and only if $n$ is even.  When $\dim SU(n)/SO(n)=(n+2)(n-1)/2$ is odd, i.e., 
when $n\equiv 0, 3\mod 4$, the involutive isometry $gK\to {}^tg^{-1}K$ is orientation reversing.    
It remains to consider the case $n\equiv 1\mod 4$. 
When $n\equiv 1\mod 2$,  the outer automorphism group of $SL(n,\mathbb{R})$ 
is generated by the Cartan involution $X\to {}^tX^{-1}$; see \cite[p. 132-133]{murakami}.  
It follows that $SL(n,\mathbb{R})/SO(n)$ and $SU(n)/SO(n)$ 
does not admit any orientation reversing isometry if $n\equiv 1 \mod 4$. 

Next consider the symmetric space $SU(2n)/Sp(n)$ (type AII), which is dual to $SU^*(2n)/Sp(n)$ where $SU^*(2n)\subset SL(2n,\mathbb{C})$ 
consists of matrices $Z$ which commute with the transformation ${}^t(z_1,\ldots, z_{2n})= {}^t(\bar{z}_{n+1},\dots, \bar{z}_{2n}, -\bar{z}_1, \ldots, -\bar{z}_n), z\mapsto {}^t(z_1, \ldots, z_{2n})\in \mathbb{C}^{2n}$.  
Again the outer automorphism 
group of $SU^*(2n)$ is generated by the Cartan involution $X\to {}^tX^{-1}$ by the work of Murakami \cite[p. 131-132]{murakami}.  Hence $SU^*(2n)/Sp(n) $ (equivalently $SU(2n)/Sp(n)$) admits an orientation reversing 
isometry if and only if its dimension $(n-1)(2n+1)$ is odd, that is, if and only if $n$ is even.

\subsection{Symmetric spaces of types G and F II} 
Borel and Hirzebruch \cite[\S18,19]{bh} have computed the Pontrjagin numbers of $G_2/SO(4)$ (type G) and of the Cayley plane $F_4/Spin(9)$  (type F II).   In particular they showed that $p_2[G_2/SO(4)]\ne 0$ and $p_4[F_4/Spin(9)]\ne 0$. 
Hence neither $G_2/SO(4)$ nor 
$F_4/Spin(9)$ admit any orientation reversing diffeomorphism.   
Also, by the formula of G. Hirsch, the Poincar\'e polynomials of $G_2/SO(4)$ and $F_4/Spin(9)$ are $1+t^4+t^8$ and $1+t^8+t^{16}$, it follows that their signatures are  $1$ (with suitable orientations).

\subsection{The table of results}\label{rptable}
We summarise the above results for type III symmetric spaces in  Table 1: 
Here OR indicates the existence of an orientation reversing 
isometry and OP indicates that every isometry is orientation 
preserving.

\[
\begin{tabular}{ |c|c|c|c|}
\hline
Type & $U/K$  &  parameter & OP/OR \\
\hline
A I & $SU(n)/SO(n)$ & $n\equiv  0, 2,3 \mod 4$ & OR \\  
     &                                           &  $n\equiv 1 \mod 4$ &  OP\\
\hline
A II &  $SU(2n)/Sp(n)$ & $2|n$ & OR\\
& &  $2|(n-1)$ & OP\\
\hline
A III  & $\mathbb{C}G_{p+q,p}$ & $2|pq$ &OP\\
&&$pq \equiv 1\mod 2$ & OR\\
\hline 
BD I&$ \wt{G}_{p+q,p}$ & $2|p, 2|q, 8|pq$ & OP\\
&& otherwise & OR\\
\hline 
D III& $SO(2n)/U(n)$ & $n\equiv 2,3 \mod 4$ & OR\\
&& $n\equiv 0,1 \mod 4$ & OP\\
\hline 
C I& $Sp(n)/U(n)$ & $n\equiv 1,2 \mod 4$& OR\\
&& $n\equiv 0,3\mod 4$ & OP\\
\hline
C II & $ \mathbb{H}G_{p,q}$& $2|pq$ or $p\ne q$ & OP \\
&& $p=q\equiv 1\mod 2$ & OR\\
\hline 
E III&$\frac{E_6}{Spin(10)\times U(1)} $& -- & OP\\
\hline
E VII&  $\frac{E_7}{E_6\times U(1)} $& -- & OR\\
\hline
F II &  $F_4/SO(9)$ & -- & OP\\
\hline 
G & $G_2/SO(4)$ & -- & OP\\

\hline 
\end{tabular}
\]
\begin{center}
Table 1. Results for irreducible symmetric spaces of Type III. 
\end{center}

\section{Maps to negatively curved manifolds} \label{rankone}

We conclude this paper with the following observations concerning maps from/to a rank-$1$ locally symmetric space.
Indeed most of our observations hold in the more general context of complete  
negatively curved manifolds. 

\subsection{Self-maps}  \label{self}
Let $M$ be a complete negatively curved manifold with finite volume and having dimension at least $3$.  
Then it is known that $Out(\pi_1(M))$ 
is finite and that $\pi_1(M)$ is co-Hopfian.  See \cite[\S 5]{belegradek}.  When $M$ is compact, the finiteness of 
the outer automorphism group also follows from \cite[\S 5.4A]{gromov87}. Thus the group of self-homotopy equivalences of $M$ is finite.  If $\pi_1(M)$ is also residually finite, then any self-map of non-zero degree is a homotopy 
equivalence by Lemma \ref{finiteindex}(ii).  

Let $f:M_1\to M_2$ be a (proper) continuous map between 
two complete negatively curved manifolds of finite volume and let $\deg(f)\ne 0$.  Then $||M_1||\ge |\deg(f)|.||M_2||$ as observed by 
Gromov \cite[p. 8]{gromov}.  Since $||M_1||>0$ is finite, we see that $|\deg(f)|$ is bounded 
above by $||M_1||/||M_2||$.   If there exists also a continuous map $g:M_2\to M_1$ with $\deg(g)\ne 0$, then 
$|\deg(g)|\le ||M_2||/||M_1||$ and it follows that $||M_1||=||M_2||$, $|\deg(f)|=1=|\deg(g)|$.  
Suppose that $M_1$ is compact.  Then so is $M_2$ and 
by a result of Sela \cite{sela} we have that the $\pi_1(M_i)$ are Hopfian.  Hence $g\circ f$  (resp. $f\circ g$) induces an automorphism of $\pi_1(M_1)$ (resp. $\pi_1(M_2)$) and so $f$ and $g$ are homotopy equivalences.  
 If $\dim(M_1)\ge 5$, then by the topological rigidity theorem 
of Farrell and Jones \cite{fj}, we deduce that $M_1$ and $M_2$ are homeomorphic.  This has already been 
observed under the hypothesis that $|\deg(f)|=1$ by Sela \cite{sela}.  \\

\subsection{Maps from manifolds with positive first Betti number.} \label{firstbetti}

The first Betti number conjecture is the assertion that every hyperbolic $3$-manifold $X$ having finite volume 
admits a finite cover whose first Betti number is positive.   This implies the virtual Haken conjecture 
and is implied by the virtually fibred conjecture, which asserts that $X$ admits a finite cover which fibres over the circle.   The virtual fibred conjecture for closed hyperbolic $3$-manifolds has been established recently by Agol \cite{agol}. 

More generally, it has been conjectured that 
any locally symmetric space $X=\Gamma\backslash G/K,$ where $G$ is a connected semisimple (non-compact) linear Lie group of real rank $1$, admits a finite cover whose first Betti number is non-zero.  In group theoretic terms, it translates into 
the assertion that the lattice $\Gamma$ admits a finite index subgroup $\Gamma_0$ whose abelianization is 
infinite.  Although this conjecture seems to be as yet unresolved,  
it is known to be true in most cases when $\Gamma$ is arithmetic.  See \cite{millson2}, \cite{borel2}, \cite{rv}, and \cite{v}.

We want to show that if $Y=\Lambda \backslash H/L$ is any locally symmetric space where $\Lambda$ is 
torsionless and $H$ is any connected semisimple linear Lie group 
and $X$ is any manifold with positive first Betti number, then the set $[X,Y]$ is infinite.  It is easily 
seen that, when $H$ is regarded as a subgroup of $GL(N,\mathbb{R})$ for some $N$, there exists a $\lambda$ 
in $\Lambda$ not all of whose eigenvalues are on the unit circle.  It follows that $\lambda^m$ and $\lambda^n$ 
are not conjugates in $GL(N,\mathbb{R})$ if $|m|\ne |n|$.  Hence $\Lambda$ has 
infinitely many distinct conjugacy classes.   In fact, it follows from the work of Brauer \cite{brauer} (cf. \cite{pyber}) that any infinite, residually finite group has infinitely many conjugacy classes.  In particular, any finitely generated infinite subgroup of $GL(N,\mathbb{R})$ has infinitely many conjugacy classes.  The following very general result shows that 
$[X,Y]$ is infinite.

\begin{proposition} Let $X$ be any connected CW complex with positive first Betti number. 
Let $Y$ be an Eilenberg-MacLane complex $K(\Lambda,1)$ where $\Lambda$ is any group that has infinitely many  
conjugacy classes.   
 Then  the set $[X,Y]$ of (free) homotopy classes of maps from $X$ to $Y$ is infinite.
\end{proposition}
\begin{proof}  Let $\Gamma=\pi_1(X)$.  
Since the spaces involved are CW complexes, and since $Y$ is a $K(\Lambda,1)$-space,  
any homomorphism $\Gamma\to \Lambda$ is induced by a continuous map $X\to Y$.  If two maps $f, g:X\to Y$ are 
(freely) homotopic, then there exist inner-automorphisms $\iota_a\in Aut(\Gamma)$ and $\iota_b\in Aut(\Lambda)$ such 
that $g_*=\iota_b\circ f_*\circ \iota_a:\Gamma\to \Lambda$.   
Let $\alpha:\Gamma\to \mathbb{Z}$ be a non-zero element of $H^1(X;\mathbb{Z})$. We may assume that $\alpha$ is surjective; choose a $\gamma_0\in \Gamma$ such that $\alpha(\gamma_0)=1$.   For any 
$\lambda\in \Lambda$, let $f_\lambda:X\to Y$ be any continuous map which induces the homomorphism 
$\theta_\lambda:=\epsilon_\lambda \circ \alpha$ where $\epsilon_\lambda:\mathbb{Z}\to \Lambda$ is defined by 
$1\mapsto \lambda$.   Suppose that $f_\lambda$ and $f_\mu$ are (freely) homotopic.   Then there exist $a\in \Gamma$ and $b\in \Lambda$, such that $\theta_\mu=\iota_b\circ \theta_\lambda\circ \iota_a$. Evaluating both sides at $\gamma_0$, we obtain $\mu=b\lambda b^{-1}$; thus $\lambda$ and $\mu$ are conjugates in $\Lambda$.  Since $\Lambda$ has infinitely many conjugacy classes it follows that 
$[X,Y]$ is infinite.    
\end{proof}

\end{document}